\documentclass[12pt, reqno]{amsart}

\usepackage{enumerate}
\usepackage{graphicx}
\usepackage{amssymb,amsmath, amsthm} 
\usepackage{booktabs}
\usepackage{hyperref}
\usepackage{ mathrsfs }
\usepackage{algorithm}

\newtheorem{alg}{Algorithm}
\newtheorem{thm}{Theorem}[section]
\newtheorem{lemma}[thm]{Lemma} 

\theoremstyle{definition}
\newtheorem{defn}[thm]{Definition}

\newtheorem{conj}[thm]{Conjecture}
\newtheorem{example}[thm]{Example}

\newtheorem{rmk}[thm]{Remark}

\newcommand{\F}{\mathbb F}

\newcommand{\Fq}{\mathbb{F}_q}

\newcommand{\PP}{\mathbb P}
\newcommand{\PPFq}{\mathbb{P}^2(\mathbb{F}_q)}
\newcommand{\ISS}{Iampolskaia, Skorobogatov, and Sorokin}
\newcommand{\ISSS}{Iampolskaia, Skorobogatov, and Sorokin }
\newcommand{\PGL}{{\operatorname{PGL}}}

\newif \ifdetails
\detailsfalse

\begin{document}

\title{Counting Arcs in Projective Planes via Glynn's Algorithm}

\author[Kaplan]{Nathan Kaplan}
\address{Department of Mathematics, University of California, Irvine, 92697}
\email{nckaplan@math.uci.edu}

\author[Kimport]{Susie Kimport}
\address{Department of Mathematics, Stanford University, Stanford, CA 94305}
\email{skimport@stanford.edu}    

\author[Lawrence]{Rachel Lawrence}
\address{Department of Mathematics, Yale University, New Haven, CT 06511}
\email{rlaw@aya.yale.edu}    

\author[Peilen]{Luke Peilen}
\address{Department of Mathematics, Yale University, New Haven, CT 06511}
\email{luke.peilen@yale.edu}    

\author[Weinreich]{Max Weinreich}
\address{Department of Mathematics, Yale University, New Haven, CT 06511}
\email{maxhweinreich@gmail.com}    

\date{May 18, 2017}

\maketitle

\begin{abstract}
An $n$-arc in a projective plane is a collection of $n$ distinct points in the plane, no three of which lie on a line. Formulas counting the number of $n$-arcs in any finite projective plane of order $q$ are known for $n \le 8$. In 1995, \ISSS counted $9$-arcs in the projective plane over a finite field of order $q$ and showed that this count is a quasipolynomial function of $q$. We present a formula for the number of $9$-arcs in any projective plane of order $q$, even those that are non-Desarguesian, deriving \ISS's formula as a special case. We obtain our formula from a new implementation of an algorithm due to Glynn; we give details of our implementation and discuss its consequences for larger arcs.
\end{abstract}

\section{Introduction}

We begin by recalling the basic definitions needed to describe the problem of counting $n$-arcs in finite projective planes.
\begin{defn} \label{projectiveplane}
A \emph{projective plane} $\Pi$ is a collection of points $\mathcal{P}$ and a collection of lines $\mathcal{L}$, where each $\ell \in \mathcal{L}$ is a subset of $\mathcal{P}$ such that:
\begin{enumerate}
\item Every two points are incident with a unique line; that is, given distinct points $p_1, p_2, \in \mathcal{P}$ there exists a unique $\ell \in \mathcal{L}$ such that $\{p_1, p_2\} \subset \ell$.

\item Every two lines are incident with a unique point; that is, given distinct lines $\ell_1, \ell_2 \in \mathcal{L}$ there exists a unique $p \in \mathcal{P}$ with $p \in \ell_1 \cap \ell_2$.

\item There exist four points such that no three of them are contained in any line.
\end{enumerate}
Let $q$ be a positive integer. We say that $\Pi$ has \emph{order} $q$ if each line contains exactly $q+1$ points, and if each point is contained in exactly $q+1$ lines. 
\end{defn}

The projective plane over a field $k$, denoted $\PP^2(k)$, gives a well-studied algebraic class of examples. Finite projective planes that are isomorphic to $\PP^2(\F_q)$ for some finite field $\F_q$ are called \emph{Desarguesian}, and all other finite projective planes are called \emph{non-Desarguesian}. For an overview of the theory of non-Desarguesian projective planes, see \cite{weibel}.

Our goal is to count special configurations of points called $n$-arcs.
\begin{defn}
An \emph{$n$-arc} in a projective plane $\Pi$ is a collection of $n$ distinct points, no three of which are collinear.
\end{defn}
Arcs are collections of points in \emph{linear general position}, a fundamental concept in classical and algebraic geometry. In an infinite projective plane, most collections of points form arcs, but in finite projective planes, interesting enumerative problems arise.

For simplicity, throughout this paper we  count \emph{ordered $n$-arcs}, that is, $n$-tuples of points that form an arc, and we often omit the adjective ordered.  The number of ordered $n$-arcs in a projective plane is equal to the number of unordered $n$-arcs in the plane multiplied by $n!$.
\begin{defn}
Let $\Pi$ be a projective plane of order $q$. Define $C_n(\Pi)$ as the number of ordered $n$-arcs of $\Pi$.  In the case where $\Pi$ is the projective plane $\PPFq$, we write $C_n(q)$ in place of $C_n(\Pi)$.
\end{defn}
For small values of $n$ we can determine $C_n(q)$ using the algebraic structure of $\PPFq$.  For example, the homography group of $\PPFq$ is $\PGL_3(\F_q)$, which acts sharply transitively on collections of four points, no three of which lie on a line (that is, ordered 4-arcs).  Therefore, 
\[
C_4(q) = |\PGL_3(\F_q)| = (q^2+q+1)(q^2+q)q^2(q-1)^2.
\]  
An ordered $4$-arc comes from choosing a point $P$, a different point $Q$, another point $R$ not on the line $\overline{PQ}$, and another point not on the union of the lines $\overline{PQ}, \overline{PR}, \overline{QR}$.  There are $q^2+q+1$ choices for $P,\ q^2+q$ choices for $Q,\ q^2$ for $R$, and $(q-1)^2$ for $S$, so exactly the same formula holds in \emph{any} projective plane $\Pi$ of order $q$.  Similar but more intricate ideas lead to polynomial formulas for all $n \le 6$. 

\begin{thm}\cite[Theorem 4.1]{glynn1988}
For any finite projective plane $\Pi$ of order $q$,
\begin{enumerate}
\item 
$ C_1(\Pi) = C_1(q)  = q^2 + q + 1$, \\
\item
$ C_2(\Pi) = C_2(q) = (q^2 + q + 1)(q^2 + q)$, \\
\item 
$ C_3(\Pi) = C_3(q) = (q^2 + q + 1)(q^2 + q)q^2$, \\
\item 
$ C_4(\Pi) = C_4(q) = (q^2 + q + 1)(q^2 + q)q^2 (q-1)^2$, \\
\item
$ C_5(\Pi) = C_5(q) = (q^2 + q + 1)(q^2 + q)q^2 (q-1)^2(q^2-5q+6)$,\\
\item 
$ C_6(\Pi) = C_6(q) = (q^2 + q + 1)(q^2 + q)q^2 (q-1)^2(q^2-5q+6)(q^2 - 9q + 21).$ 
\end{enumerate}
\end{thm}

Formulas for $C_7(\Pi)$ and $C_8(\Pi)$ are known, but are no longer given by a single polynomial in $q$ for any finite projective plane $\Pi$.  These formulas involve the number of \emph{strong realizations} of certain special configurations of points and lines.

\begin{defn}[\cite{batten}]
A \textit{linear space} is a pair of sets $(\mathcal{P}, \mathcal{L})$, the elements of which are referred to as \textit{points} and \textit{lines} respectively, such that: 
\begin{enumerate}
\item each line is a subset of $\mathcal{P}$, 
\item any two distinct points belong to exactly one line, 
\item each line contains at least two points. 
\end{enumerate}
A pair $(p,l)$ consisting of a point $p$ and a line $l$ containing $p$ is called an \emph{incidence} (or \emph{flag}) of the linear space.
\end{defn}

Linear spaces capture the basic notions of incidence geometry without reference to a particular projective plane. Thus, a linear space may be thought of as a combinatorial blueprint, the incidence data of which may or may not be satisfied by any given set of points and lines in a projective plane. For instance, every $n$-arc encapsulates the data of the linear space $(\{1,2,\ldots,n\}, T_n),$ where $T_n$ is the set of $2$-subsets of $\{1,2,\ldots,n\}$. To formalize this notion, we recall the definition of a \textit{strong realization}.

\begin{defn}[\cite{grunbaum2009}] \label{strongreal1}
For a linear space $S = (\mathcal{P}, \mathcal{L})$, a strong realization of $S$ in a projective plane $\Pi$ is a one-to-one assignment $s: \mathcal{P} \to \Pi$ such that each subset $Q$ of $\mathcal{P}$ is contained in a line of $S$ if and only if $s(Q)$ is contained in a line of $\Pi$. 
\end{defn}

A special class of linear spaces called \emph{superfigurations} play an important role in the formulas we give.

\begin{defn}[\cite{grunbaum2009}]
Let a \emph{full line} of a linear space be a line containing more than two points. A (combinatorial) \emph{$n_k$-configuration} is a linear space on $n$ points with $n$ full lines, such that each full line contains exactly $k$ points, and each point is contained in exactly $k$ full lines.  

A linear space on $n$ points with some number of full lines, not necessarily $n$, in which each full line contains \emph{at least} $k$ points and each point is contained in \emph{at least} $k$ full lines is called an \emph{$n_k$-superfiguration}.  Since every configuration is a superfiguration, but not conversely, we state all of our results and definitions in terms of superfigurations.
\end{defn} 

Throughout this paper all $n_k$-configurations we encounter will have $k=3$, so we refer to $n_3$-configurations simply as \emph{configurations}.  Similarly, we refer to $n_3$-superfigurations as \emph{superfigurations}.

\begin{rmk}\label{notationremark}
Glynn \cite{glynn1988} refers to connected superfigurations as ``variables'', and \ISSS \cite{iss} call superfigurations ``overdetermined configurations''. We have opted to call them superfigurations for the sake of consistency with Gr\"unbaum's text  on classical configurations of points and lines \cite{grunbaum2009}, and to distinguish them from other uses of the terms ``variable" and ``configuration''.
\end{rmk}

In this paper we will be concerned with the collection of all superfigurations on at most $n$ points, so we introduce a notion of what is means for two linear spaces to be `the same'.  See \cite[Section 1.7 and Lemma 2.6.2]{batten2}.
\begin{defn}
Let $S = (P,L)$ and $S' = (P',L')$ be finite linear spaces and let $f$ be a bijective function from $P$ to $P'$.  Then $f$ is an \emph{isomorphism} of linear spaces if and only if for any $\ell \in L,\ f(\ell) \in L'$.
\end{defn}
There are no superfigurations with $6$ or fewer points.  Up to isomorphism, there is a unique superfiguration with $7$ points, which is called the \emph{Fano plane}, and a unique superfiguration with $8$ points, which is called the \emph{M\"obius-Kantor} configuration.  Note that isomorphic linear spaces have the same number of strong realizations in any projective plane.  We define a counting function for strong realizations that occurs in our formulas.

\begin{defn}
For a linear space $f$, let $A_f(\Pi)$ be the number of $n$-tuples of distinct points of $\Pi$ such that the points are the image of a strong realization in $\Pi$ of a linear space that is isomorphic to $f$.
\end{defn}

With this language, we can restate our problem as follows: give a formula for $A_a(\Pi)$, where $a$ is the linear space $(\{1,\ldots,n\},T_n)$ as defined previously. It is clear from the definition that when $f$ and $g$ are isomorphic linear spaces, $A_f(\Pi) = A_g(\Pi)$.

Glynn gives formulas for the number of $7$-arcs and the number of $8$-arcs in a projective plane $\Pi$ of order $q$.  When $s$ is the Fano plane, we write $A_7(\Pi)$ in place of $A_s(\Pi)$, and when $s$ is the M\"obius-Kantor configuration, we write $A_8(\Pi)$ for $A_s(\Pi)$.

\begin{thm}\cite[Theorems 4.2 and 4.4]{glynn1988}\label{78arcs}
Let $\Pi$ be a projective plane of order $q$.  Then
\begin{enumerate}[ {(}1{)} ]
\item
$ C_7(\Pi) = (q^2 + q + 1)(q^2 + q)q^2 (q-1)^2 (q - 3)(q - 5) (q^4 - 20q^3 + 148q^2 - 468q + 498) - A_{7}(\Pi), $ \\
\item
$ C_8(\Pi) =   (q^2 + q + 1)(q^2 + q)q^2 (q-1)^2 (q-5) \\ 
 \cdot (q^7 - 43q^6 + 788q^5 - 7937q^4 + 47097q^3 - 162834q^2 + 299280q - 222960) \\ - 8(q^2 - 20q + 78)A_{7}(\Pi) + A_{8}(\Pi)$.
\end{enumerate}
\end{thm}

\begin{rmk}
Glynn gives formulas counting $7$-arcs and $8$-arcs in terms of realizations of ``variables'' $(7_3)$ and $(8_3)$ \cite{glynn1988} (see Remark \ref{notationremark}). Our formulas for $A_7(\Pi)$ and $A_8(\Pi)$ differ from these expressions by a factor of $n!$. That is, $A_7(\Pi) = 7!(7_3)$ and $A_8(\Pi) = 8!(8_3)$.  
\end{rmk}

Glynn applies these formulas to the problem of classifying finite projective planes.  Since $A_7(\Pi) \ge 0$ we see that $C_7(\Pi)$ is at most the value of the degree $14$ polynomial given in \emph{(1)}.  For $q = 6$, this degree $14$ polynomial evaluates to $-29257200$, which implies that there is no projective plane of order $6$ \cite[Corollary 4.3]{glynn1988}.  It is not clear whether formulas for $C_n(\Pi)$ for larger $n$ will have similar consequences for the classification of finite projective planes.

Theorem \ref{78arcs} shows that the quantity $C_7(\Pi) + A_7(\Pi)$ depends only on $q$ and not on the particular projective plane $\Pi$. This allows us to compare the number of $7$-arcs among all planes of a given order. Specifically, the number $C_7(\Pi)$ is maximized for fixed order when $A_7(\Pi)$ is minimized, and vice versa.

Given $7$ ordered points in $\Pi$ that give a strong realization of the Fano plane, any of the $7!$ permutations of these points also gives a different strong realization of the Fano plane in $\Pi$.  There are $168$ choices of an ordered $4$-arc from this set of points, so each Fano subplane contributes at least $168$ to $C_4(\Pi)$.  Thus, $A_7(\Pi)/7! \leq C_4(\Pi)/168$, and therefore $A_7(\Pi) \leq 30 C_4(\Pi)$.

If a plane satisfies the property that every 4-arc together with the three meeting points of its diagonals forms a Fano plane, then that plane is said to obey the \emph{Fano axiom}. Thus, a plane that obeys the Fano axiom has $A_7(\Pi) = 30C_4(\Pi)$ exactly; otherwise, $A_7(\Pi) < 30C_4(\Pi)$. Gleason's theorem shows that the only finite planes that obey the Fano axiom are Desarguesian planes of even order \cite{gleason}. Thus, for any $q = 2^r$, the Desarguesian plane has strictly the fewest $7$-arcs among planes of order $q$.

A conjecture widely attributed to Hanna Neumann states that every non-Desarguesian finite projective plane contains a Fano subplane (see, for example, \cite{tait}). Further, it is well-known that Desarguesian planes of odd order do not contain Fano subplanes. Thus, Theorem \ref{78arcs} leads to the following equivalent reformulations of Neumann's conjecture.

\begin{conj} \label{neumann1}
\begin{enumerate}[ {(}1{)} ]
\item Let $\Pi$ be a finite non-Desarguesian plane. Then 
\[
C_7(\Pi) < (q^2 + q + 1)(q^2 + q)q^2 (q-1)^2 (q - 3)(q - 5) (q^4 - 20q^3 + 148q^2 - 468q + 498).
\]

\item Let $q$ be a fixed odd prime power and let $\Pi$ be a non-Desarguesian projective plane of order $q$.  Then $C_7(\Pi) < C_7(q)$.
\end{enumerate}
\end{conj}

When $\Pi= \PPFq$, we can compute $A_7(q)$ and $A_8(q)$ to get formulas for $C_7(q)$ and $C_8(q)$. We have seen that $A_7(q) = 0$ for all odd $q$, and is given by a polynomial in $q$ when $q$ is even.
\begin{defn}
A \emph{quasipolynomial} of period $m$ is a function $g(x)$ of the positive integers such that there is a collection of polynomials $f_0(x),\ldots, f_{m-1}(x)$ satisfying $g(x) = f_i(x)$ for all $x \equiv i \pmod{m}$.
\end{defn}
Such functions are sometimes called \emph{PORC}, or \emph{polynomial on residue classes}. If the formula $A_7(q)$ for prime powers $q$ is extended to all integers we obtain a quasipolynomial of period 2, which only depends upon the parity of $q$. 
Similarly, $A_8(q)$ is a quasipolynomial of period $3$.  Glynn computes these quasipolynomials and gives explicit quasipolynomial formulas for $C_7(q)$ and $C_8(q)$ \cite{glynn1988}. He did not push this method further, noting ``the complexity of the problem as the number of points approaches $10$.''

In order to study the problem of counting inequivalent linear MDS codes, \ISSS give a formula for the number of ordered $n$-arcs in $\PPFq$.  For more on the connection between arcs and MDS codes, see \cite{hirschfeldthas}. There are $10$ superfigurations on $9$ points up to isomorphism, which we denote by $9_3,\ldots, 9_{12}$.  Let $A_{9_i}(\Pi)$ denote the number of strong realizations of $9_i$ in the projective plane $\Pi$.  When $\Pi = \PPFq$ we write $A_{9_i}(q)$ instead of $A_{9_i}(\Pi)$.
\begin{thm}\cite[Theorem 1]{iss}\label{ISSThm}
The number of $9$-arcs in the projective plane over the finite field $\Fq$ is given by
\begin{eqnarray*}
C_9(q) &=& (q^2 + q + 1)(q^2 + q)(q^2)(q-1)^2 \\
& & \bigg[q^{10} - 75q^9 + 2530q^8 - 50466q^7 + 657739q^6 - 5835825q^5 \\
& & +35563770q^4 - 146288034q^3 + 386490120q^2 - 588513120q \\
& & + 389442480  - 1080\big(q^4 - 47q^3 + 807q^2 - 5921q + 15134\big)a(q) \\ 
& &       + 840\big(9q^2 -243q + 1684\big)b(q) + 30240\big(-9c(q) + 9 d(q)+ 2e(q)\big)\bigg]
\end{eqnarray*}
where 
\begin{eqnarray*}
a(q) & = & \begin{cases}  
1 & \text{ if } 2 \mid q, \\
0 &  \text{ otherwise};
\end{cases} \\
b(q) & = &  \#\{x \in \Fq: x^2 + x + 1 = 0\}, \\
c(q) & = & \begin{cases} 
1 & \text{ if } 3 \mid q, \\ 
0 & \text{ otherwise};
\end{cases}\\
d(q) & = & \#\{x \in \Fq: x^2 + x - 1 = 0\}, \\
e(q) & = & \#\{x \in \Fq: x^2 + 1 = 0\}. \\
\end{eqnarray*}
\end{thm}
\ISSS give quasipolynomial formulas for each of the functions $a(q), b(q), c(q), d(q), e(q)$ \cite{iss}. For example, $e(q)$ depends only on $q$ modulo $4$.  Substituting these formulas into Theorem \ref{ISSThm} gives $C_9(q)$ as a quasipolynomial of period $60$.

In order to prove Theorem \ref{ISSThm}, \ISSS use the fact that there is a natural way to assign coordinates to the points of $\PPFq$.  For example, by taking an appropriate change of coordinates, every superfiguration on at most $9$ points is projectively equivalent to one where five points are chosen to be $[1:0:0],\ [0:1:0],\ [0:0:1],\ [1:1:1]$, and $[1:1:0]$.  Therefore, we need only count the number of strong realizations of the superfiguration where these five points are fixed (see \cite{iss, sturmfels}).  

We cannot generally assign coordinates to the points of a non-Desarguesian projective plane. In particular, Theorem \ref{ISSThm} is insufficient to describe the number of $9$-arcs in non-Desarguesian planes, that is, planes that are not coordinatized by a field.

The main result of this paper is to extend Theorem \ref{ISSThm} to any projective plane of order $q$.
\begin{thm}\label{9arcs}
Let $\Pi$ be a projective plane of order $q$.  The number of $9$-arcs in $\Pi$ is
\begin{eqnarray*}
C_9(\Pi) &= & q^{18} - 75q^{17} + 2529q^{16} - 50392q^{15} + 655284q^{14} \\
& & - 5787888q^{13} + 34956422q^{12} - 141107418q^{11} + 356715069q^{10} \\
& &-477084077q^9 + 143263449q^8 + 237536370q^7 + 52873326q^6 \\
& &- 2811240q^5 - 588466080q^4 + 389304720 q^3 \\
& & + (-36q^4 + 1692q^3 - 29052q^2 + 212148q - 539784) A_{7}(\Pi) \\ 
& & + (9q^2 - 243q + 1647) A_{8}(\Pi) - A_{9_{3}}(\Pi) - A_{9_{4}}(\Pi) \\
& & - A_{9_{5}}(\Pi) + A_{9_{7}}(\Pi) - 3 A_{9_{10}}(\Pi) + 3 A_{9_{11}}(\Pi) - 9 A_{9_{12}}(\Pi). \\
\end{eqnarray*}
\end{thm}

Computing $A_7(q), A_8(q)$, and $A_{9_i}(q)$ for each $i$ recovers Theorem \ref{ISSThm} as a corollary. These functions are computed in \cite{iss}, but we note that there is a minor error in the calculation of $A_{9_3}(q)$.  Despite this error, the final formula for $C_9(q)$ is correct. We include these counts for {completeness}:\\
{\renewcommand{\arraystretch}{1.2}
\resizebox{\columnwidth}{!}{
\begin{tabular}{lclll}
$A_7(q) \ = 30a(q)C_4(q)$ & & \multicolumn{3}{l}{\phantom{spacing spa}$A_{9_7}(q) \ = 60480(e(q) - a(q))C_4(q)$}\\
$A_8(q) \ = 840b(q)C_4(q)$ & &  \multicolumn{3}{l}{\phantom{spacing spa}$A_{9_8}(q) \ = 10080(q - 2 - b(q))C_4(q)$} \\
\multicolumn{3}{l}{$A_{9_3}(q) = 3360((q - 2 - b(q))(q - 5) + (q - 3)b(q))C_4(q)$} & \phantom{spa} & $A_{9_9}(q) \ = 0$ \\
$A_{9_4}(q) = 40320(q - 2 - b(q))C_4(q)$ & & & & $A_{9_{10}}(q) = 1680b(q)C_4(q)$ \\
$A_{9_5}(q) = 30240(q - 3)(1 - a(q))C_4(q)$ & & & & $A_{9_{11}}(q) = 90720d(q)C_4(q)$ \\
$A_{9_6}(q) = 30240(q - 2)a(q)C_4(q)$ & & & & $A_{9_{12}}(q) = 30240c(q)C_4(q)$.
\end{tabular}}}

There are $q^2+q+1$ points in a projective plane of order $q$, so naively counting $9$-arcs requires checking $O(q^{18})$ configurations. The formula from Theorem \ref{9arcs} reduces the problem to counting strong realizations of twelve superfigurations. Each of these superfigurations is so highly determined that at most $q^2$ sets of points must be considered following a selection of four initial points. Therefore, Theorem \ref{9arcs} reduces the total number of collections of points that we must consider to $O(q^{10})$.

Theorem \ref{9arcs} demonstrates non-obvious relationships between the function $C_9(\Pi)$ and the number of strong realizations of certain superfigurations, such as the Pappus configuration $9_3$.  Interestingly, the superfigurations $9_6, 9_8$, and $9_9$ do not influence $C_9(\Pi)$. These relationships could lead to more conjectures along the lines of Conjecture \ref{neumann1}.

In Section \ref{sectionalgorithms}, we describe the algorithm used to prove Theorem \ref{9arcs}. In Section \ref{10arcresults} we discuss computational aspects of the implementation and the difficulty of extending our results to larger arcs. Finally, in Section \ref{sec_further} we discuss related work on counting $10$-arcs in $\PPFq$.

\section{Algorithms for calculating arc formulas} \label{sectionalgorithms}

In \cite{glynn1988}, Glynn gives an inductive algorithm for counting $C_n(\Pi)$ in terms of the number of strong realizations in $\Pi$ of all superfigurations on at most $n$ points. The following statement is a version of \cite[Theorem 3.6]{glynn1988}.
\begin{thm}\label{glynn_alg_thm}
There exist polynomials $p(q), p_s(q)$ such that for any finite projective plane $\Pi$ of order $q$, we have
\begin{equation*} \label{eq:1}
C_n(\Pi) = p(q) + \sum_{s} p_s (q) A_s(\Pi),
\end{equation*}
where the sum is taken over all superfigurations $s$ with at most $n$ points up to isomorphism.
\end{thm}
We call the polynomial $g_s(q)$ the \emph{coefficient of influence} of the superfiguration $s$, since it measures the degree to which $s$ is relevant in $C_n(\Pi)$.

In order to prove Theorem \ref{9arcs} we return to Glynn's original algorithm from \cite{glynn1988}.  To our knowledge, this is the first time that the algorithm has been implemented to find new enumerative formulas, rather than just used as a theoretical tool to prove that formulas of a certain type exist.  The implementation we describe has the potential to give analogues of Theorem \ref{9arcs}, computing $C_n(\Pi)$ for larger values of $n$.

The algorithm used to arrive at Theorems 1.3 and 1.4 was first described in \cite{glynn1988}, and was further clarified by Rolland and Skorobogatov in \cite{rolland1993}. The latter form of the algorithm was a central component of the proof of Theorem \ref{ISSThm}. We present an exposition of the algorithm and prove that it works, following \cite{iss} and \cite{rolland1993}. We then discuss modifications to the algorithm that lead to a manageable runtime.

\begin{defn}[\cite{iss}]
A \textit{boolean n-function} is a function taking subsets of $\{1,2,3,\ldots,n\}$ to \{0,1\}. Two boolean $n$-functions $f$ and $g$ are \textit{isomorphic} if there is a permutation $i$ of $\{1,\ldots,n\}$ so that $g = f \circ i$.
\end{defn}

We highlight that $f$ can be thought of as \textit{labeled}: for instance, we must distinguish the boolean 2-function which just sends \{1\} to 1 from the isomorphic boolean 2-function which just sends \{2\} to 1. We also note that the boolean $n$-functions are in one-to-one correspondence with the power set of the power set of $\{1,\ldots,n\}$.

\begin{defn}[\cite{iss}]
For boolean $n$-functions $f$ and $g$, we say $f \geq g$ if $f(S) \geq g(S)$ for all $S \subseteq \{1,2,\ldots,n\}$. Note that $\geq$ is a partial order on the set of boolean $n$-functions.
\end{defn}

\begin{example}
Suppose $f\colon 2^{\{1,\ldots,7\}} \rightarrow \{0,1\}$ satisfies $f(\{1,3,4,5\}) = f(\{4,5,6\}) = 1$ and $f(S) = 0$ for every other subset.
Then $f$ is a boolean $7$-function. If we let $f'\colon 2^{\{1,\ldots, 7\}} \rightarrow \{0,1\}$ be the function that satisfies $f'(\{1,3,4\}) = f'(\{4,5,6\}) = 1$ and $f'(S) = 0$ for every other subset, then it is \emph{not} the case that $f \geq f'$.
\end{example}

Our goal is to explain an abstraction of the axioms of geometry where the elements of $\{1,\ldots,n\}$ are points and the function $f$ is an indicator function that evaluates to $1$ for the collinear subsets of $\{1,\ldots,n\}$.

\begin{defn}
A boolean $n$-function $f$ is a \textit{linear space function} if it satisfies the following:
\begin{enumerate}[ {(}1{)} ]
\item If $f(I) = 1,$ then $f(J) = 1$ for all $J \subseteq I$. \\

\item If $\#I \leq 2,$ then $f(I) = 1$. \\

\item If $f(I) = f(J) = 1$ and $\# (I \cap J) \geq 2$, then $f(I \cup J) = 1$.
\end{enumerate}
If $f$ does not satisfy these requirements, we call $f$ \textit{pathological.}
\end{defn}

\noindent In words, these requirements mean the following: 
\begin{itemize}
\item subsets of collinear sets are collinear sets; 
\item any set of $0, 1, $ or $2$ points qualifies as a collinear set; 
\item if two collinear sets intersect in at least two points, then the union of the sets is also a collinear set. 
\end{itemize}
Thus, a linear space function $f$ defines a linear space with points $\{1,2,\ldots,n\}$ and sets of collinear points defined by $f^{-1}(1)$. If $f$ defines a superfiguration, we sometimes say that $f$ is a superfiguration.

For the rest of the section, fix some projective plane $\Pi$ of order $q$. We now extend Definition \ref{strongreal1} to boolean $n$-functions.

\begin{defn} \label{strongreal} 
Let $f$ be a boolean $n$-function.  A \emph{strong realization} of $f$ is an $n$-tuple $S$ of distinct points labeled $1,2,\ldots, n$ in $\Pi$ such that the set $f^{-1}(1)$ is exactly the collection of collinear subsets of $S$. We denote the number of strong realizations of $f$ in $\Pi$ by $A_f(\Pi)$.
\end{defn}

If $f$ is pathological, then $A_f(\Pi) = 0$ by the projective plane axioms given in Definition \ref{projectiveplane}. If $f$ is a linear space function, then $f$ defines a linear space $s_f$ on the set $\{1,2,...,n\}$, where the lines are taken to be the collinear sets of $f$ which are maximal with respect to inclusion. In this situation, the value of $A_{s_f}(\Pi)$ is equal to $A_f(\Pi)$ multiplied by the number of linear space functions isomorphic to $f$.

\begin{defn} \label{weakreal}
Let $f$ be a boolean $n$-function.  A \emph{weak realization} of $f$ is an $n$-tuple $S$ of distinct points labeled $1,2,\ldots, n$ in $\Pi$ such that the set $f^{-1}(1)$ is a subset of the collection of collinear subsets of $S$.
\end{defn}

Consider the boolean $n$-function $a$ for which $a^{-1}(1)$ is the set of subsets of $\{1,\ldots,n\}$ of size 0, 1, or 2. Then $a$ is also a linear space function. Every tuple of $n$ distinct points in $\Pi$ is a weak realization of $a.$ A strong realization of $a$ is an $n$-arc. Therefore, the goal of the algorithm is to calculate $A_a(\Pi)$. We do this indirectly, by examining weak realizations and working backwards.

\begin{defn} For a boolean $n$-function $f$, define
$$B_f(\Pi) = \sum_{g \geq f} A_g(\Pi).$$
If $f$ is a linear space function, then $B_f(\Pi)$ is the number of weak realizations of $f$ in $\Pi$. If $f$ is pathological, then $B_f(\Pi)$ is still defined, although its interpretation in terms of weak realizations is less clear.
\end{defn}

Now we reproduce the method described by Rolland and Skorobogatov in \cite{rolland1993} to calculate $B_f(\Pi)$ in terms of realizations of linear space functions on fewer points.

\begin{defn}
Suppose that $f$ is a linear space function. A \textit{full line} of $f$ is a subset $S \subseteq \{1,\ldots,n\}$ with $\#S \ge 3$, so that $f(S) = 1$ and for all $T$ that properly contain $S,\ f(T) = 0$. In other words, there is no larger set of collinear points that includes $S$. We say the \textit{index} of a point $p$ of $f$ is the number of full lines of $f$ that include $p$.
\end{defn}
\noindent Note that we can completely describe a linear space function by giving its full lines.

The next lemma is closely related to a result of Glynn \cite[Lemma 3.14]{glynn1988}.
\begin{lemma} \label{nq calc lemma}
Suppose the linear space function $f$ on $n$ points has a point of index 0, 1, or 2. Then we may express $B_f(\Pi)$ as a polynomial in $q$ and the values $A_g(\Pi)$, where $g$ ranges over the linear space functions on $n - 1$ points. Further, $B_f(\Pi)$ is linear in the $A_g(\Pi)$.
\end{lemma}
\begin{proof}
Let $f$ be a linear space function with points $\{1,2,\ldots, n\}$ and suppose that the point $n$  has index $0, 1$, or $2$.  Let $f'$ be the linear space function on the points $\{1,2,\ldots, n-1\}$ inheriting collinearity data from $f$.  That is, a subset $S' \subset \{1,2,\ldots, n-1\}$ has $f'(S) = 1$ if and only if there is a subset $S \subset \{1,2,\ldots,n\}$ with $S' \subseteq S$ and $f(S) = 1$.

Every weak realization of $f$ in $\Pi$ is a strong realization of a linear space function $g \ge f'$ on $\{1,2,\ldots,n-1\}$ together with an additional point of $\Pi$, which we label $n$.  For each $g$, let $\mu(g,f)$ be the number of ways to add a point $n$ to a strong realization of $g$ in $\Pi$ to get a weak realization of $f$ in $\Pi$.  We thus get an equation
\[
B_f(\Pi) = \sum_{g \geq f'} \mu (g,f) A_g(\Pi).
\]

We complete the proof of the lemma by giving a method to determine $\mu(g,f)$ for any pair $g,f$ as a polynomial in $q$.  The method described by Rolland and Skorobogatov is subtle in the sense that two isomorphic boolean $n$-functions $g, g'$ do not necessarily satisfy $\mu(g,f) = \mu(g',f)$ \cite{rolland1993}.

Let $P_g$ be an ordered collection of $n-1$ points of $\Pi$ that gives a strong realization of $g$.  We want to determine the number of choices of the remaining $q^2+q+1-(n-1)$ points of $\Pi$ that we can add to  $P_g$ and label $n$ to get a weak realization of $f$ in $\Pi$.  We consider cases based on the index of the point $n$ in $f$.
\begin{enumerate}
\item Suppose the index of the point $n$ in $f$ is $0$.  Any point of $\Pi$ other than the points of $P_g$ can be added to $P_g$ to get a weak realization of $f$.  Therefore, $\mu(g,f) = q^2+q+1 - (n-1)$.

\item Suppose the index of the point $n$ in $f$ is $1$.  There is a unique full line $L$ of $f$ containing $n$.  Let $L'$ be the line of $f'$ containing the points of $L$ other than the point $n$. Note that $\#L' \ge 2$.  The strong realization $P_g$ determines a unique line $L_g$ of $\Pi$ corresponding to $L'$.

Any point of $L_g \setminus P_g$ can be labeled with $n$ to get a weak realization of $f$ in $\Pi$.  Therefore, $\mu(g,f) = q+1 - \#(L_g\cap P_g)$.

\item Suppose the index of the point $n$ in $f$ is $2$.  There are two distinct full lines $L_1$ and $L_2$ of $f$ that contain $n$.  Let $L_1' \subset \{1,2,\ldots, n-1\}$ be the subset of $\{1,2,\ldots, n-1\}$ contained in $L_1$ and define $L_2'$ analogously.  

It is possible that the points of $L_1'$ and the points of $L_2'$ are contained in a single full line of $g$.   In this case, this line determines a unique line $L_g$ of $\Pi$, and as in the previous case, any point of $L_g \setminus P_g$ can be labeled with $n$ to give a weak realization of $f$.  Therefore, $\mu(g,f) = q+1 - \#(L_g\cap P_g)$.

Now suppose that the points of $L_1'$ and the points of $L_2'$ are not contained in a single line of $g$.  The strong realization $P_g$ determines distinct lines $L_{g,1}$ and $L_{g,2}$ of $\Pi$ corresponding to $L_1'$ and $L_2'$, respectively.  In a weak realization of $f$, the point labeled with $n$ must lie on the intersection of the lines of $\Pi$ corresponding to $L_1$ and $L_2$.  If $L_{g,1} \cap L_{g,2} \in P_g$, then $\mu(g,f) = 0$.  Otherwise, $\mu(g,f) = 1$.

\end{enumerate}

\end{proof}

\ifdetails

Now we are in a position to state the algorithm in the form \ISSS gave it. Say we want to find $m_q (f)$ for some boolean $n$-function $f$.

\begin{alg}[Glynn; \ISS]
\begin{enumerate}
\item Find $n_q$ and $m_q$ for all boolean $1$-functions.
\item Assume we know $m_q(f)$ for all boolean $k$-functions $f$, and that all these $m_q$ are polynomials in $q$ and $m_q$ for. Use the above method to find $n_q$ for each linear space function on $k+1$ points, as a polynomial in $q$ and $m_q(s)$ for superfigurations $s$ on fewer points.
\item For each linear space function $f$ on $k+1$ points, compute $m_q (f)$ by Mobius inversion over the space of all boolean $(k+1)$-functions.
\item Repeat steps (2) and (3) until $k+1 = n$.
\end{enumerate}
\end{alg}
This algorithm has a prohibitively long runtime for $n = 10,$ because there are $2^{2^{10}}$ boolean 10-functions.  We avoid this problem as follows. We do induction on two levels rather than one. The first level is the number of points, as in the above version of the algorithm. The second level works along the partially ordered set: if we have data for all linear space functions $g > f$, we should be able to get data for $f$. This lets us focus just on linear space functions, which dramatically reduces the number of objects to consider.
\fi

Lemma \ref{nq calc lemma} explains why superfigurations arise in Glynn's algorithm and in Theorem \ref{glynn_alg_thm}. Since the method for computing $B_f(\Pi)$ only applies to those linear space functions $f$ with a point of index $0, 1$, or $2$, the algorithm cannot inductively find $B_f(\Pi)$ for linear space functions with all points of index at least $3$. In other words, the algorithm cannot determine $B_f(\Pi)$ for superfigurations.

The following algorithm inductively expresses each $A_s(\Pi)$ and $B_s(\Pi)$ in terms of the values $A_f(\Pi)$ for superfigurations $f$.

\begin{alg} \label{glynnsalg}
\hspace{0pt}
\begin{enumerate}[ {(}1{)} ]
\item Find $A_s(\Pi)$ and $B_s(\Pi)$ for the unique linear space function $s$ on $1$ point.
\item Assume that we have $A_t(\Pi)$ and $B_t(\Pi)$ for all linear space functions $t$ on $k$ points.
\item Use Lemma \ref{nq calc lemma} to find $B_t(\Pi)$ for every non-superfiguration linear space function $t$ on $k+1$ points.
\item Assume that $f$ is a linear space function on $k + 1$ points, and assume that we know $A_g(\Pi)$ for all linear space functions $g > f$. If $f$ is not a superfiguration, calculate $A_f (\Pi)$, by writing 
\[
A_f(\Pi) = B_f(\Pi) - \sum_{g > f}A_g(\Pi),
\]
and then use the expression found in (3) to express $B_f(\Pi)$ in terms of $A_s(\Pi)$ for $k$-point superfigurations $s$.
\item Repeat the previous step until we have calculated $A_f(\Pi)$ for all linear space functions $f$ on $k+1$ points.
\item Continue by induction until $k = n$.
\end{enumerate}
\end{alg}

This algorithm gives the number of strong or weak realizations of any $n$-point linear space function $L$ in a projective plane $\Pi$ as
\[
p(q) + \sum_{s \in S} p_s(q)A_s(q),
\]
where $S$ is the set of superfigurations on at most $n$ points, $A_s(q)$ is the number of strong realizations of superfiguration $s$ in $\Pi$, and $p(q)$ and the $p_s(q)$ are polynomials in $q.$  The sum in this expression involves many superfigurations $s$ that are isomorphic to each other.  Observing that isomorphic superfigurations have the same number of strong realizations in $\Pi$ allows us to group these terms together. We then get an expression where the sum is taken only over superfigurations on at most $n$ points up to isomorphism.

We implemented this algorithm in the computer algebra system Sage. Running it for up to $9$ points takes several minutes of computation time and outputs the formula for $9$-arcs from Theorem \ref{9arcs}.

\section{Counting larger arcs} \label{10arcresults}

Algorithm \ref{glynnsalg} computes $C_n(\Pi)$ in terms of $A_s(\Pi)$ for all superfigurations $s$ for any value of $n$. In particular, the formula for $10$-arcs in general projective planes is now within reach. However, we run into problems due to the complexity of the algorithm, which has runtime roughly proportional to the square of the number of linear space functions on $n$ points.

The complete list of $n$-point linear spaces can be determined by computing the list of hypergraphs on $n$ vertices under the constraints that the minimum set size is $3$ and the intersection of any two sets is of size at most 1 (see \cite{betten, oeis}). To restrict attention to superfigurations, we impose the additional condition that the minimum vertex degree is $3$. For $n \leq 11$, McKay's \textit{Nauty} software can quickly compute all such hypergraphs up to isomorphism (see \cite{mckay2014}). The first line of the table below matches the computations of Betten and Betten \cite{betten}.\\

\begin{center}
\begin{tabular} {@{}lllllll@{}}
\toprule
\multicolumn{7}{c}{\textbf{Counts of Linear Spaces}}                                                    \\ \midrule
$n$                            & 7  & 8  & 9   & 10   & 11     & 12      \\
Linear spaces on $n$ points & 24 & 69 & 384 & 5250 & 232929 & 28872973 \\
Superfigurations     & 1  & 1  & 10  & 151  & 16234  & $>$179000        \\
Configurations       & 1 & 1 & 3 & 10 & 31 & 229 \\ \bottomrule
\end{tabular}
\end{center}

The fast growth of these functions indicates the increasing difficulty of applying Algorithm \ref{glynnsalg}. In particular, we found that the prohibitively high runtime comes from of the difficulty of using Lemma \ref{nq calc lemma} to calculate so many values of $B_s(\Pi)$.
We introduce a variant of this algorithm that partially circumvents this problem.

Recall that the number of weak realizations of a $n$-arc is given by
\[
B_a(\Pi) = \sum_{g \geq a} A_g(\Pi)
\]
where $g$ ranges over all linear space functions on $n$ points. We may therefore express the strong realizations of the $n$-arc linear space function as
\begin{align*}
A_a(\Pi) &= B_a(\Pi) - \sum_{g > a} A_g(\Pi) \\
&= B_a(\Pi) - \sum_{\substack{g>a \\ g \text{ not a superfiguration}}} A_g(\Pi) - \sum_{\substack{s > a \\ s \text{ a superfiguration}}} A_s(\Pi),
\end{align*}
where the first sum ranges over linear space functions $g$ that are \textit{not} superfigurations, and the second sum ranges over superfigurations $s$ only.

Choose a linear space function $g$ that is minimal with respect to the partial order $\ge$ among the index set of the first sum; that is, there does not exist any linear space function $g'$ occurring in the first sum with $g>g'$. We say that $g$ is a \emph{minimal non-superfiguration} of this formula. Apply the substitution
\[
A_g(\Pi) = B_g(\Pi) - \sum_{h > g} A_h(\Pi).
\]
This eliminates the $A_g(\Pi)$ term from our formula, leaving only terms $A_h(\Pi)$ for $h > g$. By repeated applications of this substitution to a minimal non-superfiguration in the formula, we arrive at an expression of the form
\[
A_a(\Pi) = \sum_{ g \text{ not a superfiguration}} k(g)B_g(\Pi) + \sum_{s \text{ a superfiguration}} l(s)A_s(\Pi),
\]
where the $k(g)$ and $l(s)$ are integers.  Using the observation that isomorphic linear spaces have the same number of strong realizations in $\Pi$ we can group terms together to get a sum over linear spaces up to isomorphism.

We can now give a formula for $n$-arcs where each instance of $B_g(\Pi)$ is replaced by a polynomial in $q$ and $A_t(\Pi)$ for superfigurations $t$ on up to $n-1$ points. This substitution does not affect the coefficients of the $A_s(\Pi)$ for superfigurations $s$ on $n$ points. Therefore, the values $l(s)$, which we have already calculated, are the coefficients of influence for the $n$-point superfigurations. We state this as a lemma.

\begin{lemma}
In the formula for $n$-arcs given in Theorem \ref{glynn_alg_thm}, 
the coefficient of influence of each $n$-point superfiguration is a constant.
\end{lemma}

Let us consider the implications for $n = 10$. Of the $163$ superfigurations on up to $10$ points, $151$ are on exactly $10$ points. Therefore, the algorithm just described calculates $151$ of the $163$ coefficients of influence without finding any values of $B_g(\Pi)$.

Computing the coefficients of influence for the remaining $12$ superfigurations on at most $9$ points would be quite computationally intensive. The table below gives the coefficients of influence for each of the superfigurations ${10}_{13}, {10}_{14},\ldots, {10}_{163}$. These superfigurations are defined at the website \cite{website}, which we have created to organize information related to counting $10$-arcs.
 
The values in the table below were obtained by running an implementation of our algorithm in Sage on the ``Grace'' High Performance Computing cluster at Yale University. Thirty-two IBM NeXtScale nx360 M4 nodes each running twenty Intel Xeon E5-2660 V2 processor cores completed the parallelized algorithm in several hours.
%\newpage
\begin{center}\resizebox{\columnwidth}{!}{
\begin{tabular} {@{}clclclclclclclclclclclclc@{}}
\toprule
\multicolumn{12}{c}{\textbf{Coefficients of Influence in the Formula for $C_{10}(\Pi)$}} \\
$s$ \hspace{.7cm} & $p_s(q)$ \hspace{.2cm}  &$s$ \hspace{.7cm} & $p_s(q)$ \hspace{.2cm}  &$s$ \hspace{.7cm} & $p_s(q)$ \hspace{.2cm}  &$s$ \hspace{.7cm} & $p_s(q)$ \hspace{.2cm}  &$s$ \hspace{.7cm} & $p_s(q)$ \hspace{.2cm}  &$s$ \hspace{.7cm} & $p_s(q)$ \\\midrule
${10}_{13}$ & $27$ & ${10}_{39}$ & $-3$ & ${10}_{65}$ & $0$ & ${10}_{91}$ & $-1$ & ${10}_{117}$ & $2$ & ${10}_{143}$ & $-2$ \\
${10}_{14}$ & $27$ & ${10}_{40}$ & $-3$ & ${10}_{66}$ & $0$ & ${10}_{92}$ & $-1$ & ${10}_{118}$ & $-1$ & ${10}_{144}$ & $-2$\\
${10}_{15}$ & $27$ & ${10}_{41}$ & $-3$ & ${10}_{67}$ & $0$ & ${10}_{93}$ & $-1$ & ${10}_{119}$ & $1$ & ${10}_{145}$ & $-1$\\
${10}_{16}$ & $1$ & ${10}_{42}$ & $-3$ & ${10}_{68}$ & $0$ & ${10}_{94}$ & $-1$ & ${10}_{120}$ & $-1$ & ${10}_{146}$ & $-2$\\
${10}_{17}$ & $1$ & ${10}_{43}$ & $-3$ & ${10}_{69}$ & $0$ & ${10}_{95}$ & $-1$ & ${10}_{121}$ & $-1$ & ${10}_{147}$ & $-2$\\
${10}_{18}$ & $1$ & ${10}_{44}$ & $-3$ & ${10}_{70}$ & $0$ & ${10}_{96}$ & $-1$ & ${10}_{122}$ & $-1$ & ${10}_{148}$ & $-2$\\
${10}_{19}$ & $1$ & ${10}_{45}$ & $-3$ & ${10}_{71}$ & $0$ & ${10}_{97}$ & $-1$ & ${10}_{123}$ & $-1$ & ${10}_{149}$ & $-2$\\
${10}_{20}$ & $1$ & ${10}_{46}$ & $-3$ & ${10}_{72}$ & $-1$ & ${10}_{98}$ & $-1$ & ${10}_{124}$ & $2$ & ${10}_{150}$ & $1$\\
${10}_{21}$ & $1$ & ${10}_{47}$ & $9$ & ${10}_{73}$ & $-1$ & ${10}_{99}$ & $3$ & ${10}_{125}$ & $-1$ & ${10}_{151}$ & $9$\\
${10}_{22}$ & $1$ & ${10}_{48}$ & $9$ & ${10}_{74}$ & $-1$ & ${10}_{100}$ & $3$ & ${10}_{126}$ & $-1$ & ${10}_{152}$ & $9$\\
${10}_{23}$ & $1$ & ${10}_{49}$ & $9$ & ${10}_{75}$ & $-1$ & ${10}_{101}$ & $3$ & ${10}_{127}$ & $2$ & ${10}_{153}$ & $2$\\
${10}_{24}$ & $1$ & ${10}_{50}$ & $9$ & ${10}_{76}$ & $-1$ & ${10}_{102}$ & $3$ & ${10}_{128}$ & $3$ & ${10}_{154}$ & $1$\\
${10}_{25}$ & $1$ & ${10}_{51}$ & $9$ & ${10}_{77}$ & $-1$ & ${10}_{103}$ & $3$ & ${10}_{129}$ & $5$ & ${10}_{155}$ & $2$\\
${10}_{26}$ & $-2$ & ${10}_{52}$ & $4$ & ${10}_{78}$ & $-1$ & ${10}_{104}$ & $-1$ & ${10}_{130}$ & $4$ & ${10}_{156}$ & $4$\\
${10}_{27}$ & $-2$ & ${10}_{53}$ & $4$ & ${10}_{79}$ & $-1$ & ${10}_{105}$ & $0$ & ${10}_{131}$ & $12$ & ${10}_{157}$ & $5$\\
${10}_{28}$ & $-2$ & ${10}_{54}$ & $4$ & ${10}_{80}$ & $-1$ & ${10}_{106}$ & $-1$ & ${10}_{132}$ & $0$ & ${10}_{158}$ & $6$\\
${10}_{29}$ & $-2$ & ${10}_{55}$ & $6$ & ${10}_{81}$ & $-1$ & ${10}_{107}$ & $-1$ & ${10}_{133}$ & $0$ & ${10}_{159}$ & $18$\\
${10}_{30}$ & $-2$ & ${10}_{56}$ & $6$ & ${10}_{82}$ & $-1$ & ${10}_{108}$ & $0$ & ${10}_{134}$ & $0$ & ${10}_{160}$ & $1$\\
${10}_{31}$ & $16$ & ${10}_{57}$ & $6$ & ${10}_{83}$ & $-1$ & ${10}_{109}$ & $-1$ & ${10}_{135}$ & $0$ & ${10}_{161}$ & $1$\\
${10}_{32}$ & $-2$ & ${10}_{58}$ & $6$ & ${10}_{84}$ & $-1$ & ${10}_{110}$ & $1$ & ${10}_{136}$ & $0$ & ${10}_{162}$ & $10$\\
${10}_{33}$ & $-2$ & ${10}_{59}$ & $19$ & ${10}_{85}$ & $-1$ & ${10}_{111}$ & $2$ & ${10}_{137}$ & $-1$ & ${10}_{163}$ & $-6$\\
${10}_{34}$ & $-2$ & ${10}_{60}$ & $-8$ & ${10}_{86}$ & $-1$ & ${10}_{112}$ & $1$ & ${10}_{138}$ & $-1$ &  & \\
${10}_{35}$ & $-3$ & ${10}_{61}$ & $19$ & ${10}_{87}$ & $-1$ & ${10}_{113}$ & $2$ & ${10}_{139}$ & $0$ &  & \\
${10}_{36}$ & $-3$ & ${10}_{62}$ & $19$ & ${10}_{88}$ & $-1$ & ${10}_{114}$ & $2$ & ${10}_{140}$ & $-1$ &  & \\
${10}_{37}$ & $-3$ & ${10}_{63}$ & $-12$ & ${10}_{89}$ & $-1$ & ${10}_{115}$ & $1$ & ${10}_{141}$ & $-2$ &  & \\
${10}_{38}$ & $-3$ & ${10}_{64}$ & $-8$ & ${10}_{90}$ & $-1$ & ${10}_{116}$ & $1$ & ${10}_{142}$ & $-1$ &  & \\
\bottomrule
\end{tabular}}
\label{coeffs}
\end{center}

\section{Application to Counting 10-arcs}\label{sec_further}

We noted in the introduction that $C_n(q)$ is given by a polynomial in $q$ for all $n \le 6$, and that $C_7(q), C_8(q)$, and $C_9(q)$ have quasipolynomial formulas.  It is natural to ask how $C_n(q)$ varies with $q$ for larger fixed values of $n$.  In forthcoming work, Elkies \cite{elkies} analyzes $A_s(q)$ as a function of $q$ for each of the $10_3$-configurations and shows that several of these functions are not quasipolynomial. In a follow-up paper \cite{10Arcs}, we analyze $A_s(q)$ for each of the $151$ superfigurations on $10$ points, finding several more non-quasipolynomial examples.  As a consequence, we conclude that $C_{10}(q)$ is not quasipolynomial (see \cite{elkies,10Arcs}).
\subsection{Remarks on Superfigurations and their Coefficients of Influence}

The ten classical configurations ${10}_{16}, {10}_{17}, {10}_{18},\ldots , {10}_{25}$ all have coefficient of influence equal to $1$. The three superfigurations that contain only $9$ lines, ${10}_{13}, 10_{14}$, and ${10}_{15}$, share the coefficient of influence $27$. All of the fifteen superfigurations with coefficient of influence $0$ have either $11$ or $12$ lines. These and other runs of equal coefficients suggest that superfigurations with similar structure tend to share coefficients of influence. 

In the classical theory of configurations, every configuration has a ``dual'' configuration obtained by interchanging points and lines (see for example \cite[Section 1.3]{batten2}). Similarly, a superfiguration of $n$ points and $m$ lines is dual to another superfiguration of $m$ points and $n$ lines.  Perhaps contrary to expectation, dual superfigurations may have different coefficients of influence, as with superfigurations ${10}_{61}$ and ${10}_{63}$. Further, dual superfigurations may have different numbers of strong realizations in a projective plane $\Pi$. As an example, the superfigurations $9_6$ and ${10}_{13}$ are dual; however, $9_6$ has strong realizations in Desarguesian planes of even order greater than $2$, while ${10}_{13}$ does not \cite{website}.

A superfiguration with an equal number of points and lines may be \emph{self-dual}: that is, the superfiguration produced by interchanging points and lines is isomorphic to the original one. Our online catalogue of superfigurations shows that 41 of the 45 superfigurations with $10$ points and $10$ lines are self-dual \cite{website}. It is unclear whether to expect self-duality to be a common property for larger superfigurations.

\subsection*{Acknowledgments}
We thank the mathematics department at Yale University and the Summer Undergraduate Research at Yale (SUMRY) program for providing us with the opportunity to conduct this research. SUMRY is supported in part by NSF grant CAREER DMS-1149054. The first author is supported by NSA Young Investigator Grant H98230-16-10305 and by an AMS-Simons Travel Grant. The first author thanks Noam Elkies for helpful conversations. We thank the referee for many helpful comments.

We would also like to extend our gratitude to Sam Payne and Jos\'e Gonz\'alez for helping to organize SUMRY, and to Alex Reinking for valuable HPC support. We thank the Yale Center for Research Computing for allowing us to use the High Performance Computing resources.

\end{document}